\documentclass[12pt]{article}
\usepackage{graphicx} 
\usepackage{amsmath,amsthm,amssymb,times,mathptmx,epsfig,graphicx,color,hyperref,physics}
\usepackage[titletoc]{appendix}
\usepackage[english]{babel}
\graphicspath{{./},{./figures/},{./epsfigures/}}
\newtheorem{thm}{Theorem}

\newtheorem{remark}[thm]{Remark}

\newcommand{\R}{{\mathbb R}}

\newcommand{\D}{{\mathbb D}}
\newcommand{\E}{{\mathbb E}}

\newcommand{\ignore}[1]{}
\newcommand{\email}[1]{\href{mailto:#1}{\nolinkurl{#1}}}

\title{ Gradient Flows as Optimal Controlled Evolutions:\\ From $\R^n$ to  Wasserstein product spaces*}
\author{Yongxin Chen\thanks{
Aerospace Engineering, Georgia Institute of Technology, Atlanta, GA 30332, USA (\email{yongchen@gatech.edu}).}
\and Tryphon T. Georgiou\thanks{Mechanical and Aerospace Engineering, University of California, Irvine, CA 92697, USA (\email{tryphon@uci.edu}).}
\and Michele Pavon\thanks{Division of Science, New York University Abu Dhabi, U.A.E.\ and Mathematics Department, University of Padova, Italy (\email{michele.pavon@nyu.edu}).}}

\def\spacingset#1{\def\baselinestretch{#1}\small\normalsize}
\spacingset{1}

\begin{document}

\maketitle
\thispagestyle{empty}
\pagestyle{empty}
\begin{abstract}
%
%
We show that the continuous-time gradient descent in $\mathbb{R}^n$ can be viewed as an optimal controlled evolution for a suitable action functional; a similar result holds for stochastic gradient descent. We then provide an analogous characterization for the Wasserstein gradient flow of the (relative) entropy, with an action that mirrors the classical case where the Euclidean gradient is replaced by the {\em Wasserstein gradient} of the relative entropy. In the small-step limit, these continuous-time actions align with the Jordan–Kinderlehrer–Otto scheme. Next, we consider gradient flows for the relative entropy over a {\em Wasserstein product space} -- a study motivated by the stochastic-control formulation of {\em Schr\"odinger bridges}. We characterize the product-space steepest descent as the solution to a variational problem with two control velocities and a product-space Wasserstein gradient, and we show that the induced fluxes in the two components are equal and opposite. This framework suggests applications to the optimal control evolution of microrobotic swarms that can communicate their present distribution to the other swarm.
\end{abstract}

\section{Introduction}
We begin by showing that the continuous-time gradient  descent in $\R^n$ is the solution of a suitable Calculus of Variations problem. This elementary result is apparently new. The characterization  appears at first sight of limited use as the Lagrangian involves the norm squared of the gradient of the function to be minimized. Nevertheless, it possesses all the advantages of the Lagrangian formulation of classical mechanics \footnote{The Lagrangian of the action functional is a {\em scalar} function. The Lagrangian formulation allows one to use any convenient coordinate system. It connects symmetries to conservation laws (Noether's like theorems). It can be easily extended to more general settings as already shown in this paper.}
Moreover, this representation might turn out to be helpful for approximation problems or for theoretical investigations on continuous-time models of gradient descent compared to other methods, see e.g., \cite{RYCD,WJBLKY}. 
In particular, it might be an effective way to study {\em conserved quantities} (see footnote) \cite{DHL,FSRV}.

Suitable ``weighted" versions of this control problem allow us to derive similar results  for Newton's method and for  stochastic gradient descent, see the Appendix.   We then turn to gradient flows on Wasserstein space. As is well known, this topic, after the original applications to the long time behavior of solutions of several important PDEs, such as the rate of convergence to equilibrium, is now experiencing renewed interest because of statistical applications; see, e.g., \cite{CNR}. We focus here on minimizing a {\em relative entropy} functional. Again, we are able to characterize the gradient flow as solution to  a suitable optimal control problem for evolutions on Wasserstein space. The variational problem can be seen as a natural extension to this abstract setting of the calculus of variations problem on $\R^n$. Thus, there is a sort of {\em canonical form} of variational problems that characterize gradient flows. We finally turn to  flows on Wasserstein product spaces. We are  motivated by the stochastic control formulation of Schr\"odinger Bridges Theory; see, e.g., \cite[4.4]{CGP7}.  We derive a gradient flow in the Wasserstein product space. We find the remarkable property that the fluxes in the two components are {\em opposite}. Plugging in the ``steepest descent" into the evolution of the relative entropy we get what appears to be a new formula (\ref{REFF}): The two flows approach each other at a faster rate than that of two solutions of the same Fokker-Planck. In this case, we are also able to show that the gradient descent induced by the Wasserstein gradient of the relative entropy on the product space is the optimal evolution for a control problem, which naturally extends the previous ones. This result may be applied to the control of microrobotic swarms that can exchange information \cite{ZMAF,ZMAF2}.  
A good reference for all the topics considered in this paper is \cite{SAN}. The development of Section \ref{REWP} was first presented at the 2016 MTNS conference \cite{CGP6}. 

The outline of the paper is as follows. In the next section, we present a variational characterization of gradient descent.   In Section \ref{OMT}, we recall some fundamental results on both the static and dynamic formulation of Optimal Mass Transport. In Section \ref{FPGR}, we review the classical results on viewing the Fokker-Planck equation as a Gradient Flow in Wasserstein Space ${\cal W}_2$. This flow is shown in Section \ref{FPC} to be the optimal evolution for a suitable control problem over ${\cal W}_2$.
In Section \ref{REWP}, 
we derive a gradient flow in the Wasserstein product space.  Finally, in Section \ref{GDWPS}, we  provide a variational characterization of the gradient flow of the previous section as the optimal evolution of a control problem in ${\cal W}_2\times{\cal W}_2$. In the Appendix, we consider suitable weighted versions of the control problem leading to variational characterizations of Newton's method and of stochastic gradient descent.

\section{Gradient  descent as optimal controlled evolution}\label{GD}
\begin{thm}\label{thm1}Let $f$ be a $C^1$ function on $\R^n$. Assume that $\nabla f$ is Lipschitz continuous. Consider the corresponding {\em Gradient Flow} on $[0,T]$
\begin{equation}\label{GD} \dot{x}(t)=-\nabla f(x(t)),\quad x(0)=x_0.
\end{equation}
Let ${\cal X}$ denote the family of $C^1[0,T]$, $\R^n$-valued paths such that $x(0)=x_0$.  Let $\cal U$ be the family of continuous, $\R^n$-valued control functions on $[0,T]$.
Consider the following calculus of variations/optimal control problem:
\begin{eqnarray}\label{CV1}&&{\rm min}_{(x,u)\in{\cal X}\times{\cal U}}J(x,u)={\rm  min}_{(x,u)\in{\cal X}\times {\cal U}}\left[\int_0^T\left(\frac{1}{2}\|\nabla f(x(t))\|^2+\frac{1}{2}\|u(t)\|^2\right)dt +f(x(T))\right]\\
&&\dot{x}=u, \quad x(0)=x_0.\label{CV2}
\end{eqnarray}
Then, (\ref{GD}) is the optimal evolution for (\ref{CV1})-(\ref{CV2}).
\end{thm}

\begin{proof}   
For $S(x)$  a $C^1$ function on $\R^n$, define
\[\Lambda^S(x,u)=\int_0^Tu(t)\cdot\nabla S(x(t))dt-S(x(T))+S(x(0)).\]
Since $\Lambda^S$ is identically zero when the constraint $\dot{x}=u$ is satisfied in $[0,T]$, it is a {\em Lagrange functional} \cite{KP} for our control problem.  As the second step in the Lagrange procedure, we study the {\em unconstrained} minimization of $J+\Lambda^S$ on $\cal X\times\cal U$. For each fixed path $x\in\cal X$, the pointwise minimization of the integrand of $J+\Lambda^S$ yields
\[u_x^*(t,x(t))=-\nabla S(x(t)).\]
Plugging this back into $J+\Lambda^S$, we get
\[(J+\Lambda^S)(x,u_x^*)=\int_0^T\left[\frac{1}{2}\|\nabla f(x(t))\|^2-\frac{1}{2}\|\nabla S(x(t))\|^2\right]dt+f(x(T))-S(x(T))+S(x(0)).\]
If $S$ satisfies
\[S(x)=f(x),\]
we get $(J+\Lambda^{f})(x,u)\equiv f(x(0))$ which is constant over $\cal X$. Thus, minimization with respect to $x\in \cal X$ of $J+\Lambda^{f}$ is not necessary. Since $x^*$ defined by
\[\dot{x}^*(t)=u_x^*(t,x^*(t))=-\nabla f(x^*(t)),\quad x^*(0)=x_0
\]
satisfies the constraint (\ref{CV2}), the pair $(x^*,u^*)$ solves (\ref{CV1})-(\ref{CV2}) and $\{x^*(t);0\le t\le T\}$ is the optimal evolution. 
\end{proof}
Computing the convective derivative (derivative along the stream lines) of $f$ along the optimal path, we get
\begin{equation}\label{decrease}\frac{d f}{dt}(x^*(t))=\nabla f(x^*(t))\cdot\dot{x}^*=-\|\nabla f(x^*(t))\|^2.
\end{equation}
Thus, as is well-known, $f$ decreases along the optimal path. 

In the Appendix, considering suitable “weighted” versions of Problem (\ref{CV1})-(\ref{CV2}), we obtain results similar to Theorem \ref{thm1}  for
Newton’s method and for stochastic gradient descent.


\section{Elements of optimal mass transport theory}\label{OMT}
There exist excellent monographs and survey papers on this topic, see \cite{RR,E,Vil,AGS,Vil2,AG,OPV, S,PC,FG,CNR}, to which we refer the reader. We shall only briefly review some concepts and results that are relevant for the topics of this paper.
\subsection{The static problem}

Let $\nu_0$ and $\nu_1$ be probability measures on the measurable spaces $X$ and $Y$, respectively. Let $c:X\times Y\rightarrow [0,+\infty)$ be a measurable map with $c(x,y)$ representing the cost of transporting a unit of mass from location $x$ to location $y$.  Let ${\cal T}_{\nu_0\nu_1}$ be the family of measurable maps $T:X\rightarrow Y$ such that $T\#\nu_0=\nu_1$, namely such that $\nu_1$ is the {\em push-forward} of $\nu_0$ under $T$. Then Monge's optimal mass transport problem (OMT) is
\begin{equation}\label{monge}\inf_{T\in {\cal T}_{\nu_0\nu_1}}\int_{X\times Y} c(x,T(x))d\nu_0(x).
\end{equation}
As is well known, this problem may not be feasible, namely the family ${\cal T}_{\nu_0\nu_1}$ may be empty. This is never the case for the ``relaxed" version of the problem studied by Kantorovich in the 1940's
\begin{equation}\label{kantorovich}
\inf_{\pi\in\Pi(\nu_0,\nu_1)}\int_{X\times Y} c(x,y)d\pi(x,y)
\end{equation}
where $\Pi(\nu_0,\nu_1)$ are ``couplings" of $\nu_0$ and $\nu_1$, namely probability distributions on $X\times Y$ with marginals $\nu_0$ and $\nu_1$. Indeed, $\Pi(\nu_0,\nu_1)$ always contains the product measure $\nu_0\otimes\nu_1$. Let us specialize the Monge-Kantorovich problem (\ref{kantorovich}) to the case $X=Y=\R^N$ and $c(x,y)=\|x-y\|^2$. Then, if $\nu_1$ does not give mass to sets of dimension $\le N-1$, by Brenier's theorem \cite[p.66]{Vil}, there exists a unique optimal transport plan $\pi$ (Kantorovich) induced by a $d\nu_0$ a.e. unique map $T$ (Monge), $T=\nabla\varphi$, $\varphi$ convex, and we have 
\begin{equation}\label{optmap}\pi=(I\times\nabla\varphi)\#\nu_0, \quad \nabla\varphi\#\nu_0=\nu_1.
\end{equation}
Here $I$ denotes the identity map. Among the extensions of this result, we mention that to strictly convex, superlinear costs $c$ by Gangbo and McCann \cite{GM}.
The optimal transport problem may be used to introduce a useful distance between probability measures. Indeed, let $\mathcal P_2(\R^N)$ be the set of probability measures $\mu$ on $\R^N$ with finite second moment. For $\nu_0, \nu_1\in\mathcal P_2(\R^N)$, the Wasserstein (Vasershtein) quadratic distance is defined by
\begin{equation}\label{Wasserdist}
W_2(\nu_0,\nu_1)=\left(\inf_{\pi\in\Pi(\nu_0,\nu_1)}\int_{\R^N\times\R^N}\|x-y\|^2d\pi(x,y)\right)^{1/2}.
\end{equation}
As is well known \cite[Theorem 7.3]{Vil}, $W_2$ is a {\em bona fide} distance. Moreover, it provides a most natural way to  ``metrize" weak convergence\footnote{$\mu_k$ converges weakly to $\mu$ if $\int_{\R^N}fd\mu_k\rightarrow\int_{\R^N}fd\mu$ for every continuous, bounded function $f$.} in $\mathcal P_2(\R^N)$ \cite[Theorem 7.12]{Vil}, \cite[Proposition 7.1.5]{AGS} (the same applies to the case $p\ge 1$ replacing $2$ with $p$ everywhere). The {\em Wasserstein space} $\mathcal W_2$ is defined as the metric space $\left(\mathcal P_2(\R^N),W_2\right)$. It is a {\em Polish space}, namely a separable, complete metric space.

\subsection{The dynamic problem}

So far, we have dealt with {\em the static} optimal transport problem. Nevertheless, in \cite[p.\ 378]{BB} it is observed that ``...a continuum mechanics formulation was already implicitly contained in the original problem addressed by Monge... Eliminating the time variable was just a clever way of reducing the dimension of the problem". Thus, a {\em dynamic} version of the OMT problem was already {\em in fieri} in Gaspar Monge's 1781 {\em ``M\'emoire sur la th\'eorie des d\'eblais et des remblais"}\,! It was elegantly accomplished by Benamou and Brenier in \cite{BB} by showing that 
\begin{subequations}\label{eq:BBclassical}
\begin{eqnarray}\label{BB1}&&W_2^2(\nu_0,\nu_1)=\inf_{(\mu,v)}\int_{0}^{1}\int_{\R^N}\|v(t,x)\|^2\mu_t(dx)dt,\\&&\frac{\partial \mu}{\partial t}+\nabla\cdot(v\mu)=0,\label{BB2}\\&& \mu_0=\nu_0, \quad \mu_1=\nu_1.\label{boundary}
\end{eqnarray}\end{subequations}
Here, the flow $\{\mu_t; 0\le t\le 1\}$ varies over continuous maps from $[0,1]$ to $\mathcal P_2(\R^N)$ and $v$ over smooth fields.  In \cite{Vil2}, Villani states at the beginning of Chapter $7$ that two main motivations for the time-dependent version of OMT are
\begin{itemize}
\item a time-dependent model gives a more complete description of the
transport;
\item the richer mathematical structure will be useful later on.
\end{itemize}
We can add three further reasons:
\begin{itemize}
\item it opens the way to establish a connection with the {\em Schr\"{o}dinger bridge} problem, where the latter appears as a regularization of the former \cite{W,Mik, mt, MT,leo,leo2,CGP3,CGP4,CGP7,CGP8};
\item it allows us to view the optimal transport problem as an (atypical) {\em optimal control} problem \cite{CGP}-\cite{CGP4}.
\item In some applications such as interpolation of images \cite{CGP5} or spectral morphing \cite{JLG}, the interpolating flow is essential!
\end{itemize}
Let $\{\mu^*_t; 0\le t\le 1\}$ and $\{v^*(x,t); (x,t)\in\R^N\times[0,1]\}$ be optimal for (\ref{eq:BBclassical}). Then 
$$\mu^*_t=\left[(1-t)I+t\nabla\varphi\right]\#\nu_0, $$ with $T=\nabla\varphi$ solving Monge's problem  provides, in McCann's language, the {\em displacement interpolation} between $\nu_0$ and $\nu_1$. Then $\{\mu^*_t; 0\le t\le 1\}$ may be viewed as a constant-speed geodesic joining $\nu_0$ and $\nu_1$ in Wasserstein space (Otto). This formally endows  $\mathcal W_2$ with a ``pseudo-Riemannian" structure.
McCann discovered \cite{McCann} that certain functionals are {\em displacement convex}, that is convex along Wasserstein geodesics. This has led to a variety of applications. Following one of Otto's main discoveries \cite{JKO,O}, it turns out that a large class of PDE's may be viewed as {\em gradient flows}  on the Wasserstein space ${\cal W}_2$. This interpretation, because of the displacement convexity of the functionals, is well suited to establish uniqueness and to study energy dissipation and convergence to equilibrium. A rigorous setting in which to make sense of the Otto calculus has been developed by Ambrosio, Gigli and Savar\'e \cite{AGS,AG} for a suitable class of functionals. Convexity along geodesics in ${\cal W}_2$ also leads to new proofs of various  geometric and functional inequalities \cite{McCann}, \cite[Chapter 9]{Vil}. Finally, we mention that, when the space is not flat, qualitative properties of optimal transport can be quantified in terms of how bounds on the Ricci-Curbastro curvature affect the displacement convexity of certain specific functionals \cite[Part II]{Vil2}.

 The {\em tangent space} of $\mathcal P_2(\R^N)$ at a probability measure $\mu$, denoted by $T_{\mu}\mathcal P_2(\R^N)$ \cite{AGS} may be identified with the closure in  $L^2_{\mu}$ of the span of $\{\nabla\varphi:\varphi\in C^\infty_c\}$, where $C^\infty_c$ is the family of smooth functions with compact support. It is naturally equipped with the scalar product of $L^2_{\mu}$.
 
\section{The Fokker-Planck equation as a gradient flow in Wasserstein space}\label{FPGR}

Let us review the variational formulation of the Fokker-Planck equation as a gradient flow in Wasserstein space \cite{JKO,Vil,TGT}. Consider a physical system with  {\em phase space} $\R^n$ and with {\em Hamiltonian} ${\cal H}: x\mapsto H(x)=E_x$. The thermodynamic states of the system are given by the family ${\cal P}(\R^n)$ of probability distributions $P$ on $\R^n$ that admit the density $\rho$. On ${\cal P}(\R^n)$, we define the {\em internal energy} as the expected value of the Energy {\em observable} in state $P$
\begin{equation}
U(H,\rho)=\E_P\{\mathcal H\}=\int_{\R^n}H(x)\,\rho(x)dx=\langle H,\rho\rangle.
\end{equation}
Let us also introduce the (differential) {\em Gibbs entropy}
\begin{equation}
S(p)=-k\int_{\R^n} \log \rho(x) \rho(x)dx,
\end{equation}
where $k$ is Boltzmann's constant. $S$ is strictly concave on ${\cal P}(\R^n)$.  According to the Gibbsian postulate of classical statistical mechanics, the equilibrium
state of a microscopic system at constant absolute temperature $T$ and with Hamiltonian function $H$ is necessarily given by the Boltzmann distribution law with density
\begin{equation}\label{MB}\bar{\rho}(x)=Z^{-1}\exp\left[-\frac{H(x)}{kT}\right]
\end{equation}
where Z is the partition function.
Let us introduce the {\em Free Energy} functional $F$ defined by
\begin{equation}\label{freeen}F(H,\rho,T):=U(H,\rho)-TS(\rho).
\end{equation}
Since $S$ is strictly concave on ${\cal P}$ and $U(E,\cdot)$ is linear, it follows that $F$ is strictly convex on the state space ${\cal P}(\R^n)$. By Gibbs' variational principle, the Boltzmann distribution $\bar{\rho}$ is a minimum point of the free energy $F$ on ${\cal P}(\R^n)$. Also note the following relation between the {\em divergence} and the free energy:
\begin{eqnarray}\nonumber
\D(\rho\|\bar{\rho})&=&\int_{\R^n}\log\frac{\rho(x)}{\bar{\rho}(x)}\,\rho(x)dx\\&=&-\frac{1}{k}S(\rho)+\log Z+\frac{1}{kT}\int_{\R^n}H(x)\rho(x)dx\nonumber\\&=&\frac{1}{kT}F(H,\rho,T)+ \log Z. \nonumber
\end{eqnarray}
Since $Z$ does not depend on $\rho$, we conclude that Gibbs' principle is a trivial consequence of the fact that $\bar{\rho}$ minimizes $\D(\rho\|\bar{\rho})$ on ${\cal D}(\R^n)$.

Consider now an absolutely continuous curve $\mu_t: [t_0,t_1]\rightarrow \mathcal W_2$. Then \cite[Chapter 8]{AGS}, there exists a ``velocity field" $v_t\in L^2_{\mu_t}$ such that the following continuity equation holds on $(0,T)$
$$\frac{d}{dt}\mu_t+\nabla\cdot(v_t\mu_t)=0.$$
Suppose $d\mu_t=\rho_tdx$, so that the continuity equation
\begin{equation}\label{continuityeq}
\frac{\partial \rho}{\partial t}+\nabla\cdot(v\rho)=0
\end{equation}
holds. We want to study the free energy functional $F(H,\rho_t,T)$ or, equivalently, $\D(\rho_t\|\bar{\rho})$, along the flow $\{\rho_t; t_0\le t\le t_1\}$. Using (\ref{continuityeq}), we get
\begin{eqnarray}\label{FED}
\frac{d}{dt}\D(\rho_t\|\bar{\rho})=\int_{\R^n}\left[1+\log\rho_t+\frac{1}{kT}H(x)\right]\frac{\partial \rho_t}{\partial t}dx\\=-\int_{\R^n}\left[1+\log\rho_t+\frac{1}{kT}H(x)\right]\nabla\cdot(v\rho_t)dx.\label{Decay}
\end{eqnarray}
Integrating by parts, if the boundary terms at infinity vanish, we get
\begin{eqnarray}\nonumber
\frac{d}{dt}\D(\rho_t\|\bar{\rho})=\int_{\R^n}\nabla\left[\log\rho_t+\frac{1}{kT}H(x)\right]\cdot v \rho_tdx\\=\langle \nabla\log\rho_t+\frac{1}{kT}\nabla H(x), v\rangle_{L^2_{\rho_t}}.\nonumber
\end{eqnarray}
Thus, the Wasserstein gradient of $\D(\rho_t\|\bar{\rho})$ is
\begin{equation}\label{entropygrad}\nabla_{\mathcal W_2}\D(\rho_t\|\bar{\rho})=\nabla\log\rho_t+\frac{1}{kT}\nabla H(x)=\nabla\log\left(\frac{\rho_t}{\bar{\rho}}\right).
\end{equation}
The corresponding gradient flow is
\begin{eqnarray}\nonumber
\frac{\partial\rho_t}{\partial t}=\nabla\cdot\left[\left(\nabla\log\rho_t+\frac{1}{kT}\nabla H(x)\right)\rho_t\right]\\=\nabla\cdot\left[\frac{1}{kT}\nabla H(x)\rho_t\right]+\Delta\rho_t.\label{FPflow}
\end{eqnarray}
But this is precisely the Fokker-Planck equation corresponding to the {\em Langevin} diffusion process
\begin{equation}\label{diffusionprocess}
dX_t=-\frac{1}{kT}\nabla H(X_t)dt+\sqrt{2}dW_t
\end{equation}
where $W$ is a standard $n$-dimensional Wiener process. The process (\ref{diffusionprocess}) has the Boltzmann distribution (\ref{MB}) as invariant density. Recall that \cite[p.220]{AGS} $F(H,\rho_t,T)$ or, equivalently, $\D(\rho_t\|\bar{\rho})$ are {\em displacement convex} and therefore have  a unique minimizer.

Let us finally plug the ``steepest descent" (\ref{FPflow}) into (\ref{Decay}). We get, after integrating by parts, the well-known formula \cite{Gr} similar to (\ref{decrease}):
\begin{eqnarray} \nonumber
\frac{d}{dt}\D(\rho_t\|\bar{\rho})=\int_{\R^n}\left[1+\log\rho_t+\frac{1}{kT}H(x)\right]\frac{\partial \rho_t}{\partial t}dx\\\nonumber=\int_{\R^n}\left[1+\log\rho_t+\frac{1}{kT}H(x)\right]\nabla\cdot\left[\frac{1}{kT}\nabla H(x)\rho_t+\nabla\rho_t\right]dx\\=-\int_{\R^n}\|\nabla\log\left(\frac{\rho_t}{\bar{\rho}}\right)\|^2\rho_t dx\label{FEdecay1}\\=-\int_{\R^n}\|\nabla_{\mathcal W_2}\D(\rho_t\|\bar{\rho})\|^2\rho_t dx.\label{FEdecay2}
\end{eqnarray}
The  integral in (\ref{FEdecay1}) is sometimes called the {\em relative Fisher information} of $\rho_t$ with respect to $\bar{\rho}$ \cite[p.278]{Vil}. Notice that (\ref{FEdecay2}) is completely analogous to (\ref{decrease}).

\section{The Fokker-Planck as an optimal controlled evolution on Wasserstein space}\label{FPC}

Let 
$\cal P$ be the family of $\R^n$-valued, finite-energy diffusions on $C[0,1]$ equivalent to $\sigma^2 \mathcal W$, where $\mathcal W$ is stationary Wiener measure, see \cite[Subsection3.1]{CGP3}. For $P\in\cal P$, we denote by $\beta^P_+$ and $\beta^P_-$ the forward and the backward drifts, respectively. We also introduce the current and osmotic drifts by
  \[
  v^P=\frac{\beta^P_++\beta^P_-}{2}, \quad u^P=\frac{\beta^P_+-\beta^P_-}{2}.
  \]
  Then formula $(38)$ in  \cite{CGP3}, adapted to the case $\sigma^2\neq 1$, reads
\begin{equation}\label{dec}\D(Q\|P)=\frac{1}{2}\D(q_0\|p_0)+\frac{1}{2}\D(q_1\|p_1)+\frac{1}{2\sigma^2}E_Q\left[\int_{0}^{1}
\left(\|v^Q-v^P\|^2+
\|u^Q-u^P\|^2\right)dt\right]
\end{equation}
Recall that for any Markovian diffusion $P$ with forward drift field $b^P_+(t,x)$ and diffusion coefficient $\sigma^2$, the backward drift field is given by 
\begin{equation}\label{bdrift}b^P_-(t,x)=b^P_+(t,x)-\sigma^2\nabla\log\rho_t,
\end{equation}
where $\rho_t(x)$ is the density of $X_t$ under $P$. The {\em current} and {\em osmotic} drift fields are then given by
\begin{eqnarray}\label{current}
v^P(t,x)&=&\frac{b^P_+(t,x)+b^P_-(t,x)}{2}= b^P_+(t,x)-\frac{\sigma^2}{2}\nabla\log\rho_t,\\ u^P(t,x)&=&\frac{b^P_+(t,x)-b^P_-(t,x)}{2}= \frac{\sigma^2}{2}\nabla\log\rho_t.\label{osmotic}
\end{eqnarray}
It is immediate that the Fokker-Planck corresponding to $P$
\[
\frac{\partial\rho_t}{\partial t}+\nabla\cdot\left[b^P_+\rho_t\right]=\frac{\sigma^2}{2}\Delta\rho_t.
\]
can be written as a continuity equation:
\[
\frac{\partial\rho_t}{\partial t}+\nabla\cdot\left[v^P\rho_t\right]=0.
\]
Let $\bar{P}$ be the measure associated with $(\ref{diffusionprocess})$ starting with the Boltzmann-Gibbs marginal $(\ref{MB})$ at time $t=0$. We observe that 
\begin{equation}\label{equilv}v^{\bar{P}}=-\frac{1}{kT}\nabla H(x)-\nabla\log\bar{\rho}(x)=0.\end{equation}

We now follow \cite[Sections 4 and 5]{CGP3}. Let ${\cal P}(\rho_0)$ be the family of finite energy,  Markovian diffusion measures on $C[0,1]$ having $\rho_0$ as initial density and diffusion coefficient $\sigma^2=2$. We consider the following (trivial) Schr\"odinger Bridge Problem: 
\begin{equation}\label{problem}\mathcal P:\;\;\text {Minimize} \;\D(P\|\bar{P})\; \text{over}\; {\cal P}(\rho_0).
\end{equation}
By \cite[(19a)]{CGP3}
\[
\D(Q\|P)=\D(q_0\|p_0)+\E_Q\left[\int_0^1\frac{1}{2\sigma^2}\|\beta^Q_+-\beta^P_+\|^2dt\right].
\]
Hence, the solution is the measure in ${\cal P}(\rho_0)$ that has forward Ito differential $(\ref{diffusionprocess})$. 
By (\ref{dec}), (\ref{osmotic}) and (\ref{equilv}), observing that $\frac{1}{2}\D(\rho_0\|\bar{\rho})$ is constant over ${\cal P}(\rho_0)$, the problem is equivalent to minimizing over ${\cal P}(\rho_0)$
\[
\frac{1}{2}\D(\rho_1\|\bar{\rho})+\frac{1}{4}\E_P\left[\int_0^1\left(\|v^P\|^2+\|\nabla\log\frac{\rho_t}{\bar{\rho}}\|^2\right)dt\right].
\]
In the spirit of \cite[Section 5]{CGP3}, after multiplying by $2$ the criterion, we get  the following fluid-dynamic formulation of Problem $\mathcal P (\ref{problem})$:
\begin{subequations}\label{eq:BB}
\begin{eqnarray}\label{CON1}&&\inf_{(\rho,v)}J(\rho,v)=\inf_{(\rho,v)}\frac{1}{2}\int_{0}^{1}\int_{\R^n}\left(\|v(t,x)\|^2+\|\nabla\log\frac{\rho_t}{\bar{\rho}}\|^2\right)\rho_t dxdt+ \D(\rho_1\|\bar{\rho}),\\&&\frac{\partial \rho}{\partial t}+\nabla\cdot(v\rho)=0,\label{BB2}\\&& \rho(x,0)=\rho_0(x).\label{boundary}
\end{eqnarray}\end{subequations}
Here, the flow $\{\rho_t; 0\le t\le 1\}$ varies over continuous maps from $[0,1]$ to $\mathcal P_2(\R^n)$ and $v$ over smooth fields. {\em Notice that (\ref{eq:BB}) represents a natural extension to the infinite dimensional ${\cal W}_2$ of the problem considered in Section \ref{GD} since by (\ref{entropygrad}) $\nabla_{\mathcal W_2}\D(\rho_t\|\bar{\rho})=\nabla\log\left(\frac{\rho_t}{\bar{\rho}}\right)$!} 
\begin{thm}The gradient flow for the functional $\D(\rho_t\|\bar{\rho})$ on the Wasserstein space ${\cal W}_2$
\begin{eqnarray}&&\frac{\partial \rho^*}{\partial t}+\nabla\cdot(v^*\rho^*)=0,\label{BB2}\\&& \rho^*(x,0)=\rho_0(x),\\&&v^*=-\nabla\left(\log\frac{\rho^*_t}{\bar{\rho}}\right).
\end{eqnarray}
solves Problem (\ref{eq:BB}). 
\end{thm}
\begin{proof}
By (\ref{Decay}), and assuming that $\rho_t$ vanishes at infinity,
\begin{eqnarray}\nonumber&&\int_0^1\frac{d}{dt}\D(\rho_t\|\bar{\rho})dt-\D(\rho_1\|\bar{\rho})+\D(\rho_0\|\bar{\rho})\\&&=-\int_0^1\int_{\R^n}\left[1+\log\rho_t+\frac{1}{kT}H(x)\right]\nabla\cdot(v\rho_t)dxdt-\D(\rho_1\|\bar{\rho})+\D(\rho_0\|\bar{\rho})\nonumber\\&&\int_0^1\int_{\R^n}\nabla\left(\log\frac{\rho_t}{\bar{\rho}}\right)\cdot v\rho_t dxdt -\D(\rho_1\|\bar{\rho})+\D(\rho_0\|\bar{\rho})=\Lambda(\rho,v)\label{LF}
\end{eqnarray}
is a {\em Lagrange Functional} for problem (\ref{eq:BB}) since it is identically zero when the constraint (\ref{BB2}) is satisfied. We now consider  the {\em uncontrained} minimization of the criterion $J(\rho,v)$ of Problem (\ref{eq:BB}) plus  our Lagrange functional $\Lambda(\rho,v)$. Pointwise minimization of the integrand with respect to $v$ {\em for a fixed flow} $\{\rho_t; 0\le t\le 1\}$ yields the optimality condition 
\begin{equation}\label{optv}
v^*(\rho_t)=-\nabla\left(\log\frac{\rho_t}{\bar{\rho}}\right).
\end{equation}
Plugging this expression back into the integrand, we get 
\begin{eqnarray}\nonumber
J(\rho,v^*(\rho))=&&\int_{0}^{1}\int_{\R^n}\left(\frac{1}{2}\|\nabla\log\frac{\rho_t}{\bar{\rho}}\|^2+\frac{1}{2}\|\nabla\log\frac{\rho_t}{\bar{\rho}}\|^2-\|\nabla\log\frac{\rho_t}{\bar{\rho}}\|^2\right)\rho_t dxdt\\&&+ \D(\rho_1,\bar{\rho})-\D(\rho_1\|\bar{\rho})+\D(\rho_0\|\bar{\rho})=\D(\rho_0\|\bar{\rho})\nonumber
\end{eqnarray}
which is constant over ${\cal P}(\rho_0)$. Thus, minimization with respect to $\rho$ is not necessary. The evolution $\{\rho^*_t; 0\le t\le 1\}$ satisfying (\ref{BB2}) with $v=v^*$ and (\ref{boundary}) together with $v^*$ solves Problem (\ref{eq:BB}). 
\end{proof}
\noindent
In view of (\ref{FEdecay2}), $\D(\rho,\bar{\rho})$ is monotonically decreasing along $\{\rho^*_t; 0\le t\le 1\}$.
\begin{remark}
 {\em Notice that for the optimal evolution
 \[
 \frac{1}{2}\|v-\bar{v}\|^2=\frac{1}{2}\|\nabla\log\frac{\rho_t}{\bar{\rho}}\|^2=\frac{1}{2}\|u-\bar{u}\|^2,
 \]
 a sort of ``virial theorem".}
\end{remark}

\section{Relative entropy as a functional on Wasserstein product spaces}\label{REWP}

In the Schr\"{o}dinger bridge problem (SBP) \cite{F2}, one seeks the random evolution (a probability measure on path-space) which is closest in the relative entropy sense to a prior Markov diffusion evolution and has certain prescribed initial and final marginals $\mu$ and $\nu$. Thanks to Sanov's Theorem \cite{SANOV,DZ}, the solution to such a maximum entropy problem provides the {\em most likely} random evolution between the given marginals. As already observed by Schr\"{o}dinger \cite{S1,S2}, the problem may be reduced to a {\em static} problem which, except for the cost, resembles the Kantorovich relaxed formulation of the optimal mass transport problem (OMT). Considering that since \cite{BB} (OMT) also has a dynamic formulation, we have two problems which admit equivalent static and dynamic versions \cite{leo2}. Moreover, in both cases, the solution entails a flow of one-time marginals joining $\mu$ and $\nu$. The OMT yields a {\em displacement interpolation flow} whereas the SBP provides an {\em entropic interpolation flow}. 

Through the work of Mikami, Mikami-Thieullen, and Leonard \cite{Mik, mt, MT,leo,leo2}, we know that the OMT may be viewed as a ``zero-noise limit" of SBP when the prior is a sort of uniform measure on path space with vanishing variance. This connection has been extended to more general prior evolutions in \cite{CGP3,CGP4}. Moreover, we know that, thanks to a very useful intuition by Otto \cite{O}, the displacement interpolation flow $\{\mu_t; 0\le t\le 1\}$ may be viewed as a constant-speed geodesic joining $\mu$ and $\nu$ in Wasserstein space \cite{Vil}.  
What can be said from this geometric viewpoint of the entropic flow? It cannot be a geodesic, but can it be characterized as a curve that minimizes a suitable action? In \cite{CGP3}, we showed that this is indeed the case resorting to a time-symmetric fluid dynamic formulation of  SBP. The action features an additional term which is a {\em Fisher information functional}. Moreover, this characterization of the Schr\"{o}dinger bridge answers in one place the question posed by Carlen \cite[pp. 130-131]{Carlen}.

It has been observed since the early nineties that SBP can be turned, thanks to {\em Girsanov's theorem},  into a stochastic control problem with atypical boundary constraints, see \cite{DP,Bl,DPP,PW,FHS}. The latter has a fluid dynamic counterpart. It is therefore interesting to compare the flow associated with the uncontrolled evolution (prior) to the optimal one. In particular, it is interesting to study the evolution of the relative entropy in the {\em product} Wasserstein space  {\em over a finite time interval}. Thus, this study differs from previous work in OMT theory concerning relative entropy from an equilibrium distribution (density).
Consider now two absolutely continuous curves $\mu_t: [t_0,t_1]\rightarrow \mathcal W_2$ and $\tilde{\mu}_t: [t_0,t_1]\rightarrow \mathcal W_2$ and their velocity fields $v_t\in L^2_{\mu_t}$ and $\tilde{v}_t\in L^2_{\tilde{\mu}_t}$. Then, on $(0,T)$
\begin{eqnarray}\label{conteq1}
\frac{d}{dt}\mu_t+\nabla\cdot(v_t\mu_t)=0,\\ \frac{d}{dt}\tilde{\mu}_t+\nabla\cdot(\tilde{v}_t\tilde{\mu}_t)=0.\label{conteq2}
\end{eqnarray}
Let us suppose that $d\mu_t=\rho_t(x)dx$ and $d\tilde{\mu}_t=\tilde{\rho}_t(x)dx$, for all $t\in [t_0,t_1]$. Then (\ref{conteq1})-(\ref{conteq2}) become
\begin{eqnarray}\label{Conteq1}
\frac{\partial \rho}{\partial t}+\nabla\cdot(v\rho)=0,\\ \frac{\partial\tilde{\rho}}{\partial t}+\nabla\cdot(\tilde{v}\tilde{\rho})=0,\label{Conteq2}
\end{eqnarray}
where the fields $v$ and $\tilde{v}$ satisfy
$$\int_{\R^n}\|v(t,x)\|^2\rho_t(x)dx<\infty,\quad \int_{\R^n}\|\tilde{v}(t,x)\|^2\tilde{\rho}_t(x)dx<\infty.
$$
The differentiability of the Wasserstein distance $W_2(\tilde{\rho}_t,\rho_t)$ has been studied \cite[Theorem 23.9]{Vil2}. Consider instead  the relative entropy functional
 on $\mathcal W_2\times \mathcal W_2$ 
 \begin{align}\nonumber \D(\tilde{\rho}_t\|\rho_t)&=\int_{\R^n}h(\tilde{\rho}_t,\rho_t)dx=\int_{\R^n}\log\left(\frac{\tilde{\rho}_t}{\rho_t}\right)\tilde{\rho}_t dx,\\
 \mbox{taking } & h(\tilde{\rho},\rho)=\log\left(\frac{\tilde{\rho}}{\rho}\right)\tilde{\rho}.\nonumber
 \end{align}
Relative entropy functionals $\D(\cdot\|\gamma)$, where $\gamma$ is a fixed probability measure (density), have been studied as geodesically convex functionals on $P_2(\R^n)$, see \cite[Section 9.4]{AGS}.  Our study of the evolution of $\D(\tilde{\rho}_t\|\rho_t)$ is motivated by problems on a finite time interval such as the Schr\"{o}dinger bridge problem and stochastic control problems \cite[Section VI]{CGP6} where it is important to evaluate relative entropy on {\em two} flows of marginals.

We  get
\begin{eqnarray}\nonumber
\frac{d}{dt}\D(\tilde{\rho}_t\|\rho_t)=\int_{\R^n}\left[\frac{\partial h}{\partial \tilde{\rho}} \frac{\partial\tilde{\rho}}{\partial t}+\frac{\partial h}{\partial \rho} \frac{\partial \rho}{\partial t}\right]dx\\=\int_{\R^n}\left[\left(1+\log\tilde{\rho}_t-\log\rho_t\right)\left(-\nabla\cdot(\tilde{v}\tilde{\rho}_t\right)\right.\nonumber\\\left.+\left(-\frac{\tilde{\rho}_t}{\rho_t}\right)\left(-\nabla\cdot(v\rho_t\right)\right]dx\label{entropyevolution}
\end{eqnarray}
After an integration by parts, assuming that the boundary terms at infinity vanish, we get
\begin{eqnarray}\nonumber
\frac{d}{dt}\D(\tilde{\rho}_t\|\rho_t)&=&\int_{\R^n}\left[\nabla \log\left(\frac{\tilde{\rho}_t}{\rho_t}\right)\cdot \tilde{v}\tilde{\rho}_t-\nabla \frac{\tilde{\rho}_t}{\rho_t}\cdot v\rho_t\right]dx\\&=&\int_{\R^n}\left[\left(\begin{matrix}\nabla \log\left(\frac{\tilde{\rho}_t}{\rho_t}\right)\\-\nabla \frac{\tilde{\rho}_t}{\rho_t}\end{matrix}\right)\cdot \left(\begin{matrix}\tilde{v}\tilde{\rho}_t\\v\rho_t\end{matrix}\right)\right]dx.\label{1}
\end{eqnarray}
Notice that the last expression looks like
$$\left\langle \left(\begin{matrix}\nabla \log\left(\frac{\tilde{\rho}_t}{\rho_t}\right)\\-\nabla \frac{\tilde{\rho}_t}{\rho_t}\end{matrix}\right), \left(\begin{matrix}\tilde{v}\\v\end{matrix}\right)\right\rangle_{L^2_{\tilde{\rho}_t}\times L^2_{\rho_t}}.
$$
Thus, we identify the gradient of the functional $\D(\tilde{\rho}\|\rho)$ on $\mathcal W_2\times \mathcal W_2$ as
\begin{equation}\label{gradWasser}
\left(\begin{matrix}\nabla^1_{\mathcal W_2} \D(\tilde{\rho}\|\rho)\\\nabla^2_{\mathcal W_2} \D(\tilde{\rho}\|\rho)\end{matrix}\right)=\left(\begin{matrix}\nabla \log\left(\frac{\tilde{\rho}}{\rho}\right)\\-\nabla \frac{\tilde{\rho}}{\rho_t}\end{matrix}\right).
\end{equation}
Let us now compute the gradient flow on  $\mathcal W_2\times \mathcal W_2$ corresponding to gradient (\ref{gradWasser}). We get
\begin{equation}\label{steepest}\frac{\partial}{\partial t}\left(\begin{matrix}\tilde{\rho}_t\\\rho_t\end{matrix}\right)-\nabla\cdot\left(\begin{matrix}\tilde{\rho}_t\nabla\log\left(\frac{\tilde{\rho}_t}{\rho_t}\right)\\-\rho_t\nabla\left(\frac{\tilde{\rho}_t}{\rho_t}\right)\end{matrix}\right)=0.
\end{equation}
Since 
$$J_1=\tilde{\rho}_t\nabla\log\left(\frac{\tilde{\rho}_t}{\rho_t}\right)=\rho_t\nabla\left(\frac{\tilde{\rho}_t}{\rho_t}\right)=-J_2,
$$
we observe the remarkable property that in the ``steepest descent" (\ref{steepest}) on the product Wasserstein space the ``fluxes" are {\em opposite} and, therefore, $\frac{\partial\tilde{\rho}}{\partial t}=-\frac{\partial\rho}{\partial t}$. If we plug the steepest descent (\ref{steepest}) into (\ref{entropyevolution}), we get what appears to be a new formula
\begin{eqnarray}\nonumber\frac{d}{dt}\D(\tilde{\rho}_t\|\rho_t)=\int_{\R^n}\left[\left(1+\log\tilde{\rho}_t-\log\rho_t+\frac{\tilde{\rho}_t}{\rho_t}\right)\frac{\partial\tilde{\rho}}{\partial t}\right]dx\\\nonumber=-\int_{\R^n}\left[\|\nabla\log\left(\frac{\tilde{\rho}_t}{\rho_t}\right)\|^2 \tilde{\rho}_t+\|\nabla\left(\frac{\tilde{\rho}_t}{\rho_t}\right)\|^2\rho_t\right]dx\\-\int_{\R^n}\left[\|\nabla^1_{\mathcal W_2} \D(\tilde{\rho}_t\|\rho)\|^2 \tilde{\rho}_t+\|\nabla^2_{\mathcal W_2} \D(\tilde{\rho}_t\|\rho)\|^2\rho_t\right]dx\label{REFF0}\\=-\int_{\R^n}\left[\left(1+\frac{\tilde{\rho}_t}{\rho_t}\right)\|\nabla\log\left(\frac{\tilde{\rho}_t}{\rho_t}\right)\|^2 \tilde{\rho}_t\right]dx.\label{REFF}
\end{eqnarray}
(\ref{REFF0}) should be compared to (\ref{FEdecay2}).

Let us return to equation (\ref{1}). By multiplying and dividing by $\tilde{\rho}_t$ in the last term of the middle expression, we get
\begin{equation}\label{PT2006}
\frac{d}{dt}\D(\tilde{\rho}_t\|\rho_t)=\int_{\R^n}\left[\nabla \log\left(\frac{\tilde{\rho}_t}{\rho_t}\right)\cdot \left(\tilde{v}-v\right)\right]\tilde{\rho}_t dx\end{equation} 
which is precisely the expression obtained in \cite[Theorem III.1]{PT}.

\section{Gradient descent on Wasserstein Product Spaces as optimal controlled evolution}\label{GDWPS}

In this section, we seek to describe the steepest descent on Wasserstein product spaces of the previous section as an optimal evolution. The main motivation comes from Decision Making for Microrobotic  Swarms that  can communicate their present distribution to the other swarm \cite{ZMAF,ZMAF2}. Suppose that the goal is to have the two swarms tend, on a given time interval, to a similar distribution in the relative entropy sense: A sort of ``Entropic Barycenter". We then seek to minimize a suitable functional over $\D(\tilde{\rho}_0)\times\D(\rho_0)$, namely pairs of finite energy diffusions with the same diffusion coefficient $\sigma^2$ ($\sigma^2=1$ for simplicity), the first starting at $t=0$ at $\tilde{\rho}_0$, the second at $\rho_0$. 
We consider the problem of minimum optimal steering of two flows of  probability distributions in Wasserstein space on a given time interval.  More explicitly, we are dealing with the following fluid dynamic optimization problem:
\begin{subequations}\label{eq:BBB}
\begin{align}\label{PROD}&\inf_{(\tilde{\rho},\rho,\tilde{v},v)}\;\frac{1}{2}\int_{0}^{1}\int_{\R^n}\left[\|\tilde{v}\|^2\tilde{\rho}+\|v\|^2\rho+\|\nabla \log\left(\frac{\tilde{\rho}}{\rho}\right)\|^2\tilde{\rho}+\|\nabla \left(\frac{\tilde{\rho}}{\rho}\right)\|^2\rho\right] dxdt+ \D(\tilde{\rho}_1\|\rho_1),\\
&\frac{\partial \tilde{\rho}}{\partial t}+\nabla\cdot(\tilde{v}\tilde{\rho})=0,\label{BBB2a}\\&\frac{\partial \rho}{\partial t}+\nabla\cdot(v\rho)=0,\label{BBB2b}\\&\tilde{\rho}(x,0)=\tilde{\rho}_0(x),\label{BBB3a}\\& \rho(x,0)=\rho_0(x).\label{BBB3b}
\end{align}\end{subequations}
\begin{remark}
In view of (\ref{gradWasser}),
and in light of the fact that $\D(\tilde{\rho},\rho)$ is a ``functional of two variables,''
Problem (\ref{eq:BBB}) represents a natural extension to  ${\cal W}_2\times {\cal W}_2$ of  Problem (\ref{eq:BB}) which, in turn, was an infinite-dimensional extension of the problem in Section \ref{GD}.
\end{remark}

\begin{thm}The gradient flow for the functional $\D(\tilde{\rho}_t\|\rho_t)$ on the Wasserstein product space ${\cal W}_2\times {\cal W}_2$
\begin{eqnarray}&&\frac{\partial \tilde{\rho}^*}{\partial t}+\nabla\cdot(\tilde{v}^*\tilde{\rho}^*)=0\\&&\frac{\partial \rho^*}{\partial t}+\nabla\cdot(v^*\rho^*)=0,\\&&\tilde{\rho}^*(x,0)=\tilde{\rho}_0(x),\\&& \rho^*(x,0)=\rho_0(x),\\&&\tilde{v}^*=\nabla \log\left(\frac{\tilde{\rho}}{\rho}\right)\\&&v^*=-\nabla\left(\frac{\tilde{\rho}^*_t}{\rho_t}\right).
\end{eqnarray}
solves Problem (\ref{eq:BBB}). 
\end{thm}
\begin{proof}
We want to show that the optimal evolution for this problem, with optimal {\em feedback} controls (velocities)
\[
\left(\begin{matrix}\tilde{v}^*\\v^*\end{matrix}\right)=\left(\begin{matrix}\nabla \log\left(\frac{\tilde{\rho}}{\rho}\right)\\-\nabla \frac{\tilde{\rho}}{\rho_t}\end{matrix}\right),
\]
is the gradient descent (\ref{steepest}).  Assuming that both $\tilde{\rho}_t$ and $\rho_t$ vanish at infinity, and using (\ref{BBB2a})-(\ref{BBB2b}), we get
\begin{eqnarray}\nonumber &&0=\int_0^1\frac{d}{dt}\D(\tilde{\rho}_t\|\rho_t)dt-\D(\tilde{\rho}_1\|\rho_1)+\D(\tilde{\rho}_0\|\rho_0)\\&&=-\int_0^1\int_{\R^n}\left\{\left[1+\log\frac{\tilde{\rho}_t}{\rho_t}\right]\nabla\cdot (\tilde{v}\tilde{\rho}_t)-\frac{\tilde{\rho}_t}{\rho_t}\left(\nabla\cdot(v\rho_t)\right)\right\} dxdt-\D(\tilde{\rho}_1\|\rho_1)+\D(\tilde{\rho}_0\|\rho_0)\nonumber\\&&=\int_0^1\int_{\R^n}\left[\nabla\left(\log\frac{\tilde{\rho}_t}{\rho_t}\right)\cdot\tilde{v}\tilde{\rho}_t-\nabla\left(\frac{\tilde{\rho}_t}{\rho_t}\right)\cdot v\rho_t\right]dxdt -\D(\tilde{\rho}_1\|\rho_1)+\D(\tilde{\rho}_0\|\rho_0)\nonumber
\end{eqnarray}
This calculation shows that
\[\Lambda(\tilde{\rho},\rho,\tilde{v},v)=\int_0^1\int_{\R^n}\left[\nabla\left(\log\frac{\tilde{\rho}_t}{\rho_t}\right)\cdot\tilde{v}\tilde{\rho}_t-\nabla\left(\frac{\tilde{\rho}_t}{\rho_t}\right)\cdot v\rho_t\right]dxdt -\D(\tilde{\rho}_1\|\rho_1)+|D(\tilde{\rho}_0\|\rho_0)\]
is a {\em Lagrange Functional} for problem (\ref{eq:BBB}) since it is identically zero when the constraints (\ref{BBB2a})-(\ref{BBB2b}) are satisfied. We now consider the {\em uncontrained} minimization of $J+\Lambda$ where  $J(\tilde{\rho},\rho,\tilde{v},v)$ is the criterion of Problem (\ref{eq:BB}).  Pointwise minimization of the integrand with respect to $\tilde{v}$ and $v$ {\em for  fixed flows} $\{\tilde{\rho}_t; 0\le t\le 1\}$ and  $\{\rho_t; 0\le t\le 1\}$  yields the optimality conditions (see (\ref{gradWasser}))
\begin{equation}\label{optimal v's}
\left(\begin{matrix}\tilde{v}^*\\v^*\end{matrix}\right)=\left(\begin{matrix}-\nabla \log\left(\frac{\tilde{\rho}}{\rho}\right)\\\nabla \frac{\tilde{\rho}}{\rho_t}\end{matrix}\right)=
\left(\begin{matrix}-\nabla^1_{\mathcal W_2} \D(\tilde{\rho}\|\rho)\\-\nabla^2_{\mathcal W_2} \D(\tilde{\rho}\|\rho)\end{matrix}\right).
\end{equation}
Plugging these expressions into the integrand of  $(J+\Lambda)$, we get $(J+\Lambda)\equiv  \D(\tilde{\rho}_0\|\rho_0)$ which is constant {\em} over the admissible flows satisfying (\ref{BBB3a})-({\ref{BBB3b}). Therefore, minimization with respect to $(\tilde{\rho},\rho)$ is not necessary.  We get that the optimal evolutions in $\D(\tilde{\rho}_0)\times\D(\rho_0)$ satisfy
\begin{equation}\label{steepest2}\frac{\partial}{\partial t}\left(\begin{matrix}\tilde{\rho}^*_t\\\rho^*_t\end{matrix}\right)-\nabla\cdot\left(\begin{matrix}\tilde{\rho}^*_t\nabla\log\left(\frac{\tilde{\rho}^*_t}{\rho^*_t}\right)\\-\rho^*_t\nabla\left(\frac{\tilde{\rho}^*_t}{\rho^*_t}\right)\end{matrix}\right)=0.
\end{equation}
We conclude that  the optimal controlled evolution for Problem \eqref{eq:BBB} is indeed the steepest (gradient) descent (\ref{steepest}).}
\end{proof}
\noindent
In view of (\ref{REFF0}), $\D(\tilde{\rho},\rho)$ is monotonically decreasing along $\{(\tilde{\rho}^*_t,\rho^*_t); 0\le t\le 1\}$.

\section{Recap and outlook}

In this work, we developed a unified control perspective on gradient descent dynamics in both the Euclidean and the Wasserstein setting, and through that lens, extended to a product-space gradient flow of a relative entropy functional. Specifically, first in finite dimensions, we explained how classical gradient descent and weighted variants can be viewed as optimizers of a convex action with a terminal potential--this Lagrangian/control framing may suggest principled design choices in the cost. Next, we recast the Fokker–Planck gradient flow of the free energy as the solution of an analogous optimal control problem, highlighting the correspondence between dissipation on the side of the PDE formulation with the Euler–Lagrange structure of the control problem. Finally, we identified the gradient of the relative entropy functional on a product Wasserstein space and derived a coupled continuity system whose equal-and opposite fluxes yield steepest descent for the functional. The formalism sharpens the single-flow structure and suggests coordinated transport laws where two communicating densities meet in the middle. 

In closing, we note interest in a set of minimal assumptions that guarantee contraction in the product-space flow of the relative entropy functional as well as error estimates and convergence rates. Extensions to $N$-fold product spaces, with suitably generalized notions of divergences, would be of interest in flows of multiple interacting populations. Finally, efficient solvers for the opposite-flux system would be of value in density feedback control and swarm coordination. Potential connections to small-noise limits of stochastic bridges may once again underscore links between entropy regularization, Fisher information, and control costs, as in the now classical stochastic control formulation of Schr\"dinger bridges \cite{CGP3}.

\begin{appendix}

\section{Weighted gradient descents}
In this appendix, we consider ``weighted" versions of the calculus of variations problem of Section \ref{GD}.
\subsection{Newton's method}

A minor modification of what we have seen in the previous subsection allows to derive Newton's descent as an optimal controlled evolution. Let $f$ be a $C^2$ function on $\R^n$ that we would like to minimize. Let $H_{f}(x)$ denote the {\em Hessian} of $f$ at $x$. Assume $H_{f}(x)$ to be positive definite on $\R^n$. As before, let ${\cal X}$ denote the family of $C^1$, $\R^n$-valued paths such that $x(0)=x_0$.  Let $\cal U$ be the family of continuous, $\R^n$-valued control functions. Consider the problem
\[{\rm  min}_{(x,u)\in{\cal X}\times {\cal U}}\left[\int_0^T\left(\frac{1}{2}\nabla f(x(t))'H_{f}(x(t))^{-1}\nabla f(x(t))+\frac{1}{2} u(t)'H_{f}(x(t))^{-1}u(t)\right)dt +f(x(T))\right]\]
subject to
\[\dot{x}=u, \quad x(0)=x_0.\]
For $S(x)$  a $C^1$ function on $\R^n$, let $\Lambda^S$ be as in the previous subsection.  The {\em unconstrained} minimization of $J+\Lambda^S$ on $\cal X\times\cal U$, for each fixed path $x\in\cal X$, now yields
\[u_x^*(t,x(t))=-H_{f}(x(t))^{-1}\nabla S(x(t)).\]
Plugging this back into $J+\Lambda^S$, we get
\begin{eqnarray}\nonumber(J+\Lambda^S)(x,u_x^*)&=&\int_0^T\left[\frac{1}{2}\nabla f(x(t))'H_{f}(x(t))^{-1}\nabla f(x(t))-\frac{1}{2}\nabla f(x(t))'H_{f}(x(t))^{-1}\nabla f(x(t))\right]dt\\&+&f(x(T))-S(x(T))+S(x(0)).
\end{eqnarray}
If $S$ satisfies
\[S(x)=f(x)\]
$(J+\Lambda^{f})(x,u)\equiv f(x(0))$ which is constant over $\cal X$. Thus, minimization with respect to $x\in \cal X$ of $J+\Lambda^{f}$ is not necessary. Let $x^*$ satisfy
\begin{eqnarray} \dot{x}^*(t)=-H_{f}(x(t))^{-1}\nabla f(x^*),\\x^*(0)=x_0.
\end{eqnarray}
Then, the pair $(x^*,u^*)$ solves the control problem. Now we compute the convective derivative (derivative along stream lines) of $f$ along the optimal path. We get
\[\frac{d f}{dt}=\nabla f\cdot\dot{x}^*=-\nabla f^TH_{f}(x^*(t))^{-1}\nabla f.\]
Thus, $f$ decreases along the optimal path and $ \dot{x}^*(t)=-H_{f}(x^*(t))^{-1}\nabla f(x^*(t))$ is the Newton descent flow for $f$ starting at $x_0$ and ending at $x^*(T)$.

\subsection{Stochastic gradient descent}

Let $\{\Pi(t,\omega); 0\le t\le T\}$ be a stochastic process taking values in $n\times n$, diagonal matrices with a fixed (say $32$) number of ones and $n-32$ zeros on the main diagonal.
Let ${\cal X}$ denote the family of stochastic processes with  $C^1$, $\R^n$-valued paths on $[0,T]$ such that $x(0)=x_0, a.s.$  Let $\cal U$ be the family of random $\R^N$-valued control functions with continuous trajectories on $[0,T]$.  Let $f$ be the $C^1$ function on $\R^N$ with Lipschitz gradient that we would like to minimize. For  vectors $u,v\in\R^n$ and $M$ symmetric, nonnegative definite, we write $\langle u,v\rangle_M$ instead of $\langle Mu,v\rangle$. Similarly for the norm.  
Consider  problem\[{\rm min}_{(x,u)\in{\cal X}\times {\cal U}}J(x,u)={\rm  min}_{(x,u)\in{\cal X}\times {\cal U}}\E\left\{\int_0^T \left[\frac{1}{2}\|\nabla f(x(t))\|_{\Pi(t)}^2+\frac{1}{2}\|u(t)\|_{\Pi(t)}^2\right]dt +f(x(T))\right\}\]subject to\[\dot{x}(t)=\Pi(t)u(t), \quad x(0)=x_0 \;\;{\rm a.s.}\]
As Lagrange functional, we can now take
\[\Lambda^S(x,u)=\E\left\{\int_0^T\langle u,\nabla S\rangle_{\Pi(t)}(t,x)dt-S(x(T))+S(x(0))\right\}.\]
For a fixed process $x\in\cal X$, the pointwise minimization of the integrand of $J+\Lambda^S$ yields
\begin{equation}\label{optcond}
\Pi(t)u^*(t)=-\Pi(t)\nabla S(x(t)), \;{\rm a.s.}
\end{equation}
Replacing $\Pi(t)u(t)$ with $-\Pi(t)\nabla S(x(t))$ in $J+\Lambda^S$, we get
\[
J+\Lambda^S= \E\left\{\int_0^T \left[\frac{1}{2}\|\nabla f(x(t))\|_{\Pi(t)}^2-\frac{1}{2}\|\nabla S(x(t)\|_{\Pi(t)}^2\right]dt +f(x(T))-S(x(T))+S(x(0))\right\}
\]
Choosing $S(x)=f(x)$, $J+\Lambda^f$ becomes the constant $f(x(0))$ over ${\cal X}$. Then $x^*$ satisfying
\begin{eqnarray} \label{SGD}\dot{x}^*(t)=-\Pi(t)\nabla f(x^*),\\x^*(0)=x_0.
\end{eqnarray}
together with any $u^*\in{\cal U}$ satisfying (\ref{optcond}) solves the optimal control problem. Notice that $(\ref{SGD})$ represents the continuous time counterpart of a stochastic gradient descent with $\Pi(t)$ determining at each time the {\em mini-batch}. Finally, observe that
\[
\frac{d f}{dt}(x^*(t))=\nabla f (x^*(t))\cdot\dot{x}^*(t)=-\|\nabla f(x^*(t))\|_{\Pi(t)}^2.
\]
This shows once more that the values of $f$ are decreasing along the optimal path.
\end{appendix}

\end{document}